\documentclass{amsart}
\usepackage{amsthm,amscd,amssymb,verbatim,epsf,amsmath,eurosym,amsfonts,mathrsfs,graphicx, mdframed, csquotes}
\usepackage[inline]{enumitem}
\usepackage[colorlinks=true,linkcolor=blue,citecolor=blue]{hyperref}
\DeclareMathOperator{\spn}{span}
\DeclareMathOperator{\sgn}{sgn}
\DeclareMathOperator{\Tr}{Tr}
\DeclareMathOperator{\Det}{Det}

\begin{document}
\theoremstyle{plain}
\newtheorem{Thm}{Theorem}
\newtheorem{Cor}[Thm]{Corollary}
\newtheorem{Ex}[Thm]{Example}
\newtheorem{Con}[Thm]{Conjecture}
\newtheorem{Main}{Main Theorem}
\newtheorem{Lem}[Thm]{Lemma}
\newtheorem{Prop}[Thm]{Proposition}

\theoremstyle{definition}
\newtheorem{Def}[Thm]{Definition}
\newtheorem{Note}[Thm]{Note}
\newtheorem{Question}[Thm]{Question}

\theoremstyle{remark}
\newtheorem{notation}[Thm]{Notation}
\renewcommand{\thenotation}{}

\errorcontextlines=0
\numberwithin{equation}{section}
\numberwithin{Thm}{section}

\title[Conformal Mean Value Theorem]
   {A Conformal Mean Value Theorem for Solutions of the Ultrahyperbolic Equation}
\author{Guillem Cobos}
\email{guillem.cobos93@gmail.com}
\author{Brendan Guilfoyle}
\address{
          School of STEM\\
          Munster Technological University, Kerry \\
          Clash \\
          Tralee  \\
          Co. Kerry \\
          Ireland.}
\email{brendan.guilfoyle@mtu.ie}
\begin{abstract} 

Asgeirsson's theorem establishes a mean value property for solutions of the ultrahyperbolic equation. In the case of four variables, it states that the integrals of a solution over certain pairs of conjugate circles are the same. In this paper, the invariance of the four dimensional ultrahyperbolic equation under conformal maps of the pseudo-Euclidean space of signature 2+2 is used to get the most general version of the mean value theorem. 

The name non-degenerate conjugate conics is used for the most general pairs of curves over which solutions of the ultrahyperbolic equation enjoy the mean value property. These are proven to be pairs of conic sections, so that, in addition to conjugate circles which were known to exist, the picture is completed by finding mean value theorems over conjugate hyperbolae, conjugate parabolae, and line-empty pairs. 

In addition, Fritz John established a link between conjugate circles and the two rulings of a hyperboloid of revolution. The line incidence property of doubly ruled surfaces is used to prove a one-to-one correspondence between non-degenerate conjugate conics and pairs of rulings of doubly ruled surfaces in Euclidean 3-space.
\end{abstract}

\date{\today}
\maketitle

%\tableofcontents
%%%%%%%%%%%%%%%%%%%%%%%%%%%%%%%%%%%%%%%%%%%%%%%%%%%%%%%%%%%%%%%%%%%%%%
\section{Introduction and results}
%%%%%%%%%%%%%%%%%%%%%%%%%%%%%%%%%%%%%%%%%%%%%%%%%%%%%%%%%%%%%%%%%%%%%%

A differentiable function $u(x_1,x_2,x_3,x_4)$ of four real variables is said to be a solution of the ultrahyperbolic equation if it satisfies

\begin{equation}\label{UHE}
\frac{\partial^2 u}{\partial x_1^2} + \frac{\partial^2 u}{\partial x_2^2} -
\frac{\partial^2 u}{\partial x_3^2} -
\frac{\partial^2 u}{\partial x_4^2} = 0.
\end{equation}

This equation is the compatibility condition for a function on line space to be the integral over lines of a function on ${\mathbb R}^3$. The X-ray transform, in which a real-valued function on ${\mathbb R}^3$ is mapped to its integral over lines, is thus overdetermined, and the equation arises in tomography where one seeks to invert the X-ray transform on line data \cite{conebeam} \cite{fourierrebinning} \cite{defrisetliu} \cite{lichen}. 

In a previous paper \cite{us}, the authors considered the flat ultrahyperbolic equation (\ref{UHE}) as the Laplace equation of the canonical conformally flat neutral metric on the space of oriented lines, utilizing the ambient conformal geometry to analyze its solutions. Conformal transformations of this metric correspond to projective transformations of ${\mathbb R}^3$, thus linking the geometries in the two spaces. While many of these geometric ideas can be extended to higher dimension, the case of 4-dimensions is of particular interest, both for its physical significance and mathematical elegance.

As a partial differential equation, the ultrahyperbolic equation is neither elliptic nor hyperbolic and many standard PDE techniques originating in the standard Laplace equation or wave equation do not apply. Asgeirsson \cite{LA} showed, however, that solutions of the ultrahyperbolic equation do satisfy a mean value property. In particular,

\begin{equation} \label{asgeirsson}
\int_0^{2\pi} u(a+r\cos\theta, b+r\sin\theta, c, d) \ d\theta
= 
\int_0^{2\pi} u(a,b,c+r\cos\theta, d+r\sin\theta) \ d\theta,
\end{equation}
for all $a,b,c,d\in\mathbb R$ and $r>0$. The mean value theorem itself has applications in image reconstruction \cite{lightfieldcoordinates}.
The intention of this paper is to determine the most general pair of curves over which the mean value theorem holds for solutions of the ultrahyperbolic equation. This extends both the work of Fritz John \cite{fjohn} and of the authors \cite{us}. 

John \cite{fjohn} claims that one can obtain the most general form of Asgeirsson's theorem by applying a linear map leaving the neutral quadratic form 
\[
||\underline x||^2 = x_1^2+x_2^2-x_3^2-x_4^2,
\]
invariant. However, we will show that this can in fact be improved by the use of {\em conformal} maps of the associated pseudo-Riemmanian metric $g=dx_1^2+dx_2^2-dx_3^2-dx_4^2$. Throughout the notation $\mathbb R^{2,2}$ will be used for the pseudo-Riemannian manifold $(\mathbb R^4, g)$.

Setting $a=b=c=d=0$ and $r=1$ in  equation (\ref{asgeirsson}) we obtain a pair of circles 
\begin{align*}
S_0 & = \{ (\alpha_1, \alpha_2,0,0) \ | \ \alpha_1^2 + \alpha_2^2 =1 \}\ \\
S_0^\perp & = \{ (0,0, \beta_1, \beta_2) \ | \ \beta_1^2 + \beta_2^2 =1 \},
\end{align*}
referred to as \emph{standard conjugate circles}. By applying conformal maps to the pair of curves $S_0$ and $S_0^\perp$ one obtains new pairs of curves over which a mean value theorem holds for solutions of the ultrahyperbolic equation. Following John, such pairs of curves arising as the conformal image of $S_0,S_0^\perp$ are called \emph{non-degenerate conjugate conics}. 

Given a pair of non-degenerate conjugate conics $S,S^\perp$ there always exists a unique pair of affine $2$-planes $\pi,\pi^\perp\subset \mathbb R^{2,2}$ such that $S$ is contained in $\pi$ and $S^\perp$ is contained in $\pi^\perp$. We provide a classification of  non-degenerate conjugate conics depending on the signature of the induced metric on the containing planes by the pseudo-Riemmanian metric $g$ \cite[Definition 3.2]{us}. Here a plane is said to be {\it definite}, {\it indefinite} or {\it degenerate} according to whether the metric induced on the plane by the ambient metric is definite, indefinite or degenerate, respectively.

The following notation is used to refer to the hypersphere in $\mathbb R^{2,2}$ centered at $\underline a\in\mathbb R^{2,2}$ and with radius-square $\pm r^2$,
$$H_{\underline a, \pm r^2}:=\{\underline x\in\mathbb R^{2,2} \ | \ ||\underline x-\underline a||^2=\pm r^2\}.$$
The notation $C_{\underline a}$ is used to refer to the hypersphere of zero radius $r=0$ with center at $\underline a\in\mathbb R^{2,2}$.
With that, all pairs of non-degenerate conjugate conics are listed in the following result:

\vspace{0.1in}

\begin{Thm}[Classification of non-degenerate conjugate conics according to the containing planes] \label{equivdefs} Let $S,S^\perp$ be two non-degenerate conjugate conics. Then they correspond to either one of the following: 
\begin{itemize}
\item {\bf Case 1}. \emph{Line-Empty Pairs}: $S$ is a line and $S^\perp$ is at infinity (or vice versa).
$$
\begin{cases}
$S$ \text{ is a line.}\\
S^\perp = \emptyset
\end{cases}
$$
\item {\bf Case 2}. $S$ and $S^\perp$ are contained in complementary affine planes $\pi$ and $\pi^\perp$, respectively.
\begin{enumerate}[label=(\roman*)]
\item \emph{Circles} in definite planes:
$$
\begin{cases}
S =\pi \cap H_{\underline p,  \pm r^2}\\
S^\perp =\pi^\perp \cap H_{\underline p, \mp r^2},
\end{cases}
$$
where $\pi,\pi^\perp$ are definite affine planes in $\mathbb R^{2,2}$ intersecting at $\underline p\in\mathbb R^{2,2}$ and $r^2\neq 0$. 
\item \emph{Hyperbolae} in indefinite planes:
$$
\begin{cases}
S =\pi \cap H_{\underline p,  \pm r^2}\\
S^\perp =\pi^\perp \cap H_{\underline p, \mp r^2},
\end{cases}
$$
where $\pi,\pi^\perp$ are indefinite affine planes in $\mathbb R^{2,2}$ intersecting at $\underline p\in\mathbb R^{2,2}$ and $r^2\neq 0$.

\item \emph{Parabolae} in degenerate planes:
$$
\begin{cases}
S =\pi \cap C_{\underline {\tilde p}}\\
S^\perp =\pi^\perp \cap C_{\underline p},
\end{cases}
$$
where $\pi,\pi^\perp$ are non-intersecting degenerate affine planes in $\mathbb R^{2,2}$ and $\underline p\in\pi, \underline {\tilde p}\in\pi^\perp$ are null-separated points. 
\end{enumerate}
\end{itemize}
\end{Thm}
\vspace{0.1in}

The proof relies on the characterization of non-degenerate conjugate conics as intersections of hyperspheres in $\mathbb R^{2,2}$. Through the use of polyspherical coordinates of hyperspheres it is possible to identify a pair of non-degenerate conjugate conics $S,S^\perp$ with a pair of complementary indefinite linear $3$-dimensional subspaces $V,V^\perp$ in a six-dimensional space of neutral signature $(3,3)$, denoted $\mathbb R^{3,3}$. 

The planes $\pi,\pi^\perp\subset \mathbb R^{2,2}$ containing the conics are then identified with the intersections of $V\cap \mathcal P$ and  $V^\perp\cap \mathcal P$ respectively, where $\mathcal P$ is a special degenerate $5$-subspace of $\mathbb R^{3,3}$ containing the polyspherical coordinates of hyperplanes in $\mathbb R^{2,2}$. The different possibilities for the signature of the metric induced on the spaces $V\cap \mathcal P, V^\perp\cap \mathcal P$ make up the different cases for the non-degenerate conjugate conics.

Altogether, one obtains the conformal version of Asgeirsson's mean value theorem for solutions of the ultrahyperbolic equation.

\vspace{0.1in}
\begin{Thm}\label{t:1} 
Let $u$ be a solution of the ultra-hyperbolic equation. Let $S,S^\perp$ be a pair of non-degenerate conjugate conics in ${\mathbb R}^{2,2}$. Whenever the integrals are defined, we have
\[
\int_S u \ dl = \int_{S^\perp} u \ dl,
\]
where $dl$ is the line element induced on the curves $S,S^\perp$ by the flat metric $g$.
\end{Thm}
\vspace{0.1in}

In \cite{fjohn} the first link was made between the ultrahyperbolic equation and straight lines in space. In particular, the variables $\underline x = (x_1, x_2, x_3, x_4)$ are interpreted as local coordinates in the manifold of straight lines through Pl\"ucker coordinates \cite[pg. 303]{fjohn}, and solutions of equation (\ref{UHE}) are shown to be line integrals of functions on $\mathbb R^3$. In other words, the ultrahyperbolic equation is the compatibility condition that a line function arising as a line integral of a scalar function in $\mathbb R^3$ must satisfy.

Asgeirsson's relation (\ref{asgeirsson}) has a special interpretation when $(x_1,x_2,x_3,x_4)$ are seen as line coordinates in the way proposed by John. The curves $S_0,S_0^\perp$ over which this mean value property holds correspond to the two families of generating lines of the one-sheeted hyperboloid of revolution 
$H_0\subset{\mathbb R}^3$ given by $x^2 + y^2 = \frac{1}{4}(1+z^2)$. Similarly, non-degenerate conjugate conics are shown to correspond to the pairs of families of lines of non-planar doubly ruled surfaces in $3$-space.

\vspace{0.1in}
\begin{Thm} \label{rulsurf}
Let $S,S^\perp$ be two curves in $\mathbb R^{2,2}$ representing the two one-parameter families of lines $L,L^\perp$ in $3$-space. Then $S,S^\perp$ are a pair of non-degenerate conjugate conics in $\mathbb R^{2,2}$ if and only if $L$ and $L^\perp$ are the two families of generating lines of a non-planar doubly ruled surface in $3$-space.
\end{Thm}
\vspace{0.1in}

In Section 2, it is shown that a solution of the ultrahyperbolic equation has the mean value property over any pair of curves $S,S^\perp$ that arise as the image of $S_0,S_0^\perp$ under an arbitrary conformal mapping of $\mathbb R^{2,2}$. This proves Theorem \ref{t:1}.

Section 3 contains a formal treatment of conjugate conics as objects in the conformal geometry of the pseudo-Euclidean space $\mathbb R^{2,2}$. A complete list of all non-degenerate conjugate conics is given, proving Theorem \ref{equivdefs}.

In Section 4, the one-to-one correspondence between pairs of non-degenerate conjugate conics and pairs of families of generating lines of doubly ruled surfaces in $3$-space is proven, as per Theorem \ref{rulsurf}. 
\vspace{0.1in}

\section{Conformal extension of Asgeirsson's mean value theorem}
%%%%%%%%%%%%%%%%%%%%%%%%%%%%%%%%%%%%%%%%%%%%%%%%%%%%%%%%%%%%%%%%%%%%%%

Define $S_0,S_0^{\perp}\subset{\mathbb R}^{2,2}$ to be the pair of circles 
$$
S_0:=\{(\alpha_1,\alpha_2,0, 0) \in {\mathbb R}^{2,2} \ | \ \alpha_1^2 +\alpha_2^2 = 1\},
$$
and
$$
S_0^\perp:=\{(0,0,\beta_1,\beta_2) \in{\mathbb R}^{2,2} \ |\ \beta_1^2 +\beta_2^2 = 1\}.
$$
In this dimension the original mean value theorem of Asgeirsson can be stated:
\vspace{0.1in}

\begin{Thm}[Asgeirsson's Theorem \cite{LA}] \label{Asgeirsson}
Let $u:{\mathbb R}^{2,2}\rightarrow \mathbb R$ be a solution of $\Delta u=0$. Then the following integral equation holds
$$
\int_{S_0} u \ dl = \int_{S_0^\perp} u \ dl,
$$
where $dl$ represents the line element induced by the flat metric $g=dx_1^2+dx_2^2-dx_3^2-dx_4^2$.
\end{Thm}
\vspace{0.1in}

Asgeirsson also proves a similar result over pairs of circles whose center is not necessarily the origin and of arbitrary radius. John mentions in \cite{fjohn} that an Asgeirsson's theorem in its most general form can be obtained by applying isometries of the neutral metric to those circles. However, in the present paper a more general way of extending Asgeirsson's theorem is considered through the use of conformal isometries of the metric.

The flat Laplace operator $\Delta$ is the ultrahyperbolic operator which acts on twice continuously differentiable functions $u:{\mathbb R}^{4}\rightarrow \mathbb R$ as
$$
\Delta u :=
\frac{\partial^2 u}{\partial x_1^2} + \frac{\partial^2 u}{\partial x_2^2} -
\frac{\partial^2 u}{\partial x_3^2} -
\frac{\partial^2 u}{\partial x_4^2}.
$$

The ultrahyperbolic equation (\ref{UHE}) can be seen as the Laplace equation $\Delta u =0$ in the pseudo-Riemmanian manifold $\mathbb R^{2,2} \equiv (\mathbb R^4, g)$. It is well-known that solutions of the Laplace equation show conformal invariance, in the sense that if $f:\mathbb R^{2,2}\rightarrow\mathbb R^{2,2}$ is a conformal diffeomorphism 
$$
g(df(v), df(w)) =\lambda g(v,w) \qquad v,w\in T\mathbb R^{2,2},
$$
with conformal factor $\lambda:\mathbb R^{2,2}\rightarrow\mathbb R$, then $\lambda (u\circ f)$ solves the Laplace equation if and only if $u$ does \cite[Proposition 2.4]{us}.

Conformal invariance of the ultrahyperbolic equation (\ref{UHE}) can be exploited to extend the mean value relation stated originally by Asgeirsson:
\vspace{0.1in}

\begin{Prop}\cite[Proposition 2.12]{us} \label{conformalasgeirsson}
Let $u:{\mathbb R}^{2,2}\rightarrow \mathbb R$ be a solution of $\Delta u=0$ and $f:{\mathbb R}^{2,2} \rightarrow {\mathbb R}^{2,2}$ be a conformal map. Then

$$
\int_{f(S_0)} u \ dl = \int_{f(S_0^\perp)} u \ dl,
$$
where $dl$ represents the line element induced by the flat metric $g$.
\end{Prop}
\vspace{0.1in}

The purpose of the present paper is to the determine the most general image of the circles $S_0,S_0^\perp$ under a conformal map of $\mathbb R^{2,2}$. Any pair of curves $S,S^\perp$ arising in this manner is called a pair of \emph{non-degenerate conjugate conics}, and according to Proposition \ref{conformalasgeirsson}, a mean value relation for solutions of equation (\ref{UHE}) will hold over it. Once these have been classified as in Theorem \ref{equivdefs}, Theorem \ref{t:1} follows.

\begin{comment}
\begin{Thm} \label{ndccs1} A pair of \emph{non-degenerate conjugate conics} in in $\mathbb R^{2,2}$ is one of the listed below:
\begin{enumerate}[label=\roman*.]
\item \emph{A line and nothing}
\item \emph{Conjugate circles}, a pair of circles with a common center and opposite radius-square contained in a pair of complementary definite $2$-planes. 
\item \emph{Conjugate hyperbolas}, a pair of pseudo-circles with a common center and opposite radius-square contained in a pair of complementary indefinite $2$-planes. 
\item \emph{Conjugate parabolas}, a pair of parabolas obtained as the intersection of two non-intersecting degenerate $2$-planes with the isotropic cones whose centers are null-separated points in either planes.
\end{enumerate}
\end{Thm}
\end{comment}

%%%%%%%%%%%%%%%%%%%%%%%%%%%%%%%%%%%%%%%%%%%%%%%%%%%%%%%%%%%%%%%%%%%%%%
\section{Pseudo-conformal geometry of $\mathbb R^{2,2}$}
%%%%%%%%%%%%%%%%%%%%%%%%%%%%%%%%%%%%%%%%%%%%%%%%%%%%%%%%%%%%%%%%%%%%%%
Consider the affine space $\mathbb R^{2,2}$ endowed with the bilinear form
$$
\langle \underline x, \underline x'\rangle := x_1x_1'+x_2x_2'-x_3x_3'-x_4x_4',
$$
for all $\underline x,\underline x'\in\mathbb R^{2,2}$ and associated quadratic form of signature $(2,2)$ 
$$||\underline x||^2 := \langle \underline x, \underline x\rangle =  x_1^2+x_2^2-x_3^2-x_4^2,$$
for all $\underline x\in\mathbb R^{2,2}$.  

It is well-known that the conformal geometry of the pseudo-Euclidean space $\mathbb R^{p,q}$ is better understood through its embedding into the null-cone in $\mathbb R^{p+1,q+1}$ \cite{cft}. As we will be working with the spaces $\mathbb R^{2,2}$ and $\mathbb R^{3,3}$ at the same time, we will use different notations to distinguish between the two spaces. Vectors in $\mathbb R^{2,2}$ will be denoted using underline notation, and boldface notation will be used for vectors in $\mathbb R^{3,3}$. Moreover, in a slight abuse of notation, a vector $\mathbf s\in\mathbb R^{3,3}$ will often be represented in the way $\mathbf s = (s_0, \underline s, s_5)$, where $\underline s\in\mathbb R^{2,2}$ and $s_0,s_5\in\mathbb R$.
\vspace{0.1in}

%%%%%%%%%%%%%%%%%%%%%%%%%%%%%%%%%%%%%%%%%%%%%%%%%%%%%%%%%%%%%
\subsection{Diagonal polyspherical coordinates of hyperspheres}

Hyperspheres are the fundamental objects of conformal geometry, in the sense that they are the most general subsets preserved under conformal maps. A  hypersphere can either be: 
\begin{enumerate}[label=(\roman*).]
\item a \emph{proper hypersphere} given by an equation of the form $||\underline x - \underline a ||^2 = \pm r^2$
\item a \emph{hyperplane} defined by $\langle \underline a, \underline x\rangle =b$
\end{enumerate}
These objects can be defined in a unified way by an equation of the form
\begin{equation}\label{firsthyp}
\mathsf s_0||\underline{x}||^2 + 2\langle \underline{\mathsf s}, \underline{x}\rangle  + 2 \mathsf s_5= 0.
\end{equation}
The coefficients $\underline{\mathsf s}=(\mathsf s_1, \mathsf s_2,\mathsf s_3,\mathsf s_4)\in\mathbb R^6$ are called the {\it polyspherical coordinates} of the hypersphere (\ref{firsthyp}). Polyspherical coordinates uniquely determine a hypersphere, but they are defined up to scaling, as multiplying equation (\ref{firsthyp}) by a non-zero factor leaves the solution set invariant. Further information on polyspherical coordinates of hyperspheres can be found in \cite{russian}.

Knowing the polyspherical coordinates of a hypersphere, one can find its center and radius-square as follows: 
\[
\underline a = - \frac{\underline{\mathsf s}}{\mathsf s_0}, \qquad\qquad
\pm r^2  = \frac{||\underline{\mathsf s}||^2 - 2\mathsf s_0 \mathsf s_5}{\mathsf s_0^2}.
\]

In this paper, a linear transformation of the polyspherical coordinates will be considered so that the expression on top of the radius-square is in diagonal form, i.e. \emph{sums of squares}. 
\vspace{0.1in}
\begin{Def} Let $H$ be the hypersphere in $\mathbb R^{2,2}$ given by the equation 
\begin{equation}\label{hyperspheredef}
(s_0+s_5) ||\underline{x}||^2 - 2\langle \underline{s}, \underline{x}\rangle  + (-s_0+s_5)= 0.
\end{equation}
The coefficients $\mathbf s:=(s_0, \underline{s}, s_5) = (s_0, s_1,s_2,s_3,s_4, s_5)\in\mathbb R^6$ are called the \emph{diagonal polyspherical coordinates (DPC)}  of the hypersphere $H$. Sometimes the notation $\mathbf s_H$ will be used to emphasize the relation between the coordinates and the hypersphere.
\end{Def}
\vspace{0.1in}

As noted before, the solution set of equation (\ref{hyperspheredef}) is invariant under rescaling of the numbers $(s_0,\underline s, s_5$, thus the diagonal polyspherical coordinates of a hypersphere in $\mathbb R^{2,2}$ are regarded as points in the projective space $\mathbb RP^5$. It is useful to work with the \emph{projective normalization condition} 
\begin{equation} \label{projcond}s_0+s_5=1, 
\end{equation}
as it will be shown now. If $\mathbf s:=(s_0, \underline{s}, s_5)$ are the normalized DPC of a hypersphere in $\mathbb R^{2,2}$ (so that $s_0+s_5=1$), the center and radius-square have simple expressions in terms of the DPC:
\[
\underline a =  \underline s, \qquad\qquad
\pm r^2  = s_0^2+||\underline s||^2 - s_5^2. 
\]
\begin{comment}
\begin{align*}
\underline a &= \frac{\underline s}{s_0+s_5} \\
\pm r^2 & = \frac{s_0^2+||\underline s||^2 - s_5^2}{(s_0+s_5)^2}. 
\end{align*}
\end{comment}

Consider the neutral bilinear form in $\mathbb R^6$ given by 
\begin{equation}\label{threethree}
(\mathbf s, \mathbf s'):=s_0s_0' + s_1s_1'+s_2s_2'-s_3s_3'-s_4s_4'- s_5s_5'.
\end{equation}
The associated quadratic form of neutral signature $(3,3)$
\[
(\mathbf s,\mathbf s) =s_0^2 + ||\underline{s}||^2 - s_5^2,
\]
is the expression of the radius-square of a hypersphere. Thus, the six-dimensional space of polyspherical coordinates of hyperspheres in $\mathbb R^{2,2}$ is naturally endowed with a bilinear structure of neutral signature $(3,3)$. Call $\mathbb R^{3,3}$ the linear space $\mathbb R^6$ endowed with this neutral bilinear form.  
\vspace{0.1in}

\begin{Ex} \label{standardhyp}
The \emph{proper hypersphere} centered at $\underline p=(p_1,p_2,p_3,p_4)$ and radius-square $\pm r^2$ of equation 
$$
H_{\underline p, \pm r^2}: ||\underline x- \underline p||^2 = \pm r^2,
$$
has polyspherical coordinates
$$
(s_0, \underline{s},s_5) = \left(
\frac{1-||\underline p||^2\pm r^2}{2},p_1,p_2,p_3,p_4, \frac{1+||\underline p||^2\mp r^2}{2} \right).
$$

The \emph{standard positive hypersphere} in $\mathbb R^{2,2}$ is given by $||\underline{x}||^2= 1$,
and has diagonal polyspherical coordinates
$$
(s_0, \underline{s},s_5)=(1,0,0,0,0,0).
$$
The \emph{standard negative hypersphere} given by $||\underline{x}||^2 = -1$ has DPC
$$
(s_0, \underline{s},s_5)=(0,0,0,0,0,1).
$$
\end{Ex}
\vspace{0.1in}

\begin{Ex}\label{hypplane} A hyperplane in $\mathbb R^{2,2}$
$$
\langle \underline a, \underline x \rangle := a_1x_1 + a_2x_2-a_3x_3-a_4x_4 = b,
$$
has diagonal polyspherical coordinates
$$
(s_0, \underline{s},s_5)=(-b, a_1,a_2,a_3,a_4, b).
$$
\end{Ex}
\vspace{0.1in}
\begin{Def} 
The proper hypersphere  $C_{\underline p}$ centered at $\underline p\in\mathbb R^{2,2}$ 
$$
||\underline x - \underline p||^2 = 0,
$$
and of zero radius-square
is called the \emph{isotropic cone} centered at $\underline p$.
A hyperplane $H$ given by equation $\langle \underline a, \underline x \rangle = b$ is said to be a \emph{null hyperplane} whenever $||\underline a||^2 = 0$.
\end{Def}
\begin{Ex}
The isotropic cone $C_{\underline p}$ has diagonal polyspherical coordinates
\begin{equation}\label{eq:DPScone}
\mathbf s_{\underline p} = \left(
\frac{1-||\underline p||^2}{2},p_1,p_2,p_3,p_4, \frac{1+||\underline p||^2}{2} \right).
\end{equation}
The null hyperplane $\langle \underline a, \underline x\rangle = b$ has diagonal polyspherical coordinates
\[
\left(
-b,a_1,a_2,a_3,a_4,b
\right), \quad ||\underline a||^2 =0.
\]
\end{Ex}
\vspace{0.1in}

The diagonal polyspherical coordinates of null hyperplanes and isotropic cones both satisfy the quadratic condition
$s_0^2 + ||\underline s||^2 - s_5^2=0$. Moreover, any vector $\mathbf s = (s_0, \underline s, s_5)\in\mathbb R^{3,3}$ satisfying this nullity condition and for which $\underline s \neq \underline 0$, will correspond to the diagonal polyspherical coordinates of an isotropic cone or a null hyperplane.
\vspace{0.1in}

\begin{Def}\label{nullconeQ} Define $Q\subset \mathbb R^{3,3}$ to be the null-cone in $\mathbb R^{3,3}$, i.e. 
$$
Q := \{ \mathbf s  \in  \mathbb R^{3,3}\ | \ (\mathbf s, \mathbf s) = 0 \},
$$
and define the following subset 
\begin{align*}
Q_1 & = Q \cap \{s_0 + s_5 = 1\}. 
%Q_0 &  = Q \cap \{s_0+s_5 = 0\} \\
\end{align*}
\end{Def}
\vspace{0.1in}

\begin{Prop} The set of isotropic cones in $\mathbb R^{2,2}$ is in one-to-one correspondence with the points in $Q_1$. 
%The set of null hyperplanes in $\mathbb R^{2,2}$ is in one-to-one correspondence with 
Moreover, the map 
$\phi: \mathbb R^{2,2} \rightarrow Q_1$ defined by $\phi(\underline x) = \mathbf s_{\underline x}$ as in equation (\ref{eq:DPScone}) is a bijection.
\end{Prop}
\vspace{0.1in}

\begin{comment}
\begin{Note} 
Moreover, if $(s_0, \underline{s},s_5)$ are diagonal polyspherical coordinates of a  hypersphere $H$ satisfying condition (\ref{nullcone33}) and $s_0+s_5\neq0$, then $H$ is the isotropic cone $C_{\underline a}$ with center $\underline a\in\mathbb R^{2,2}$ given by 
$$
\underline a = \frac{\underline s}{s_0+s_5}.
$$\end{Note}
\end{comment}

%%%%%%%%%%%%%%%%%%%%%%%%%%%%%%%%%%%%%%%%%%%%%%%%%%%%%%%%%%%%%%%%%%%%%%
\subsection{From DPC to the points that form the hypersphere}\label{DPCtopoints}
\vspace{0.1in}

\begin{Prop} \label{dpstopoints} Let $\underline x\in\mathbb R^{2,2}$ and $H\subset \mathbb R^{2,2}$ a hypersphere with diagonal polyspherical coordinates $\mathbf s_H\in \mathbb R^{3,3}$. Then $\underline x\in H$ if and only if $(\mathbf s_{\underline x}, \mathbf s_H)=0$. 
\end{Prop}

Using Proposition \ref{dpstopoints} we can find the points that belong to a hypersphere $H$ in $\mathbb R^{2,2}$ from the diagonal polyspherical coordinates $\mathbf s = (s_0, \underline{s}, s_5)\in \mathbb R^{3,3}$ of $H$. Given the diagonal polyspherical coordinates $\mathbf s = (s_0, \underline{s}, s_5)\in \mathbb R^{3,3}$ of a hypersphere $H$, take the \emph{complementary hyperplane} in $\mathbb R^{3,3}$ of equation
\begin{equation}\label{polarplane} \ s_0x_0 + \langle \underline{s},\underline{x}\rangle -s_5x_5 =0,
\end{equation}
and intersect it with $Q_1$. The preimage of the resulting set under $\phi$ is going to be the set of points in $\mathbb R^{2,2}$ that form $H$.

\vspace{0.1in}

\begin{Ex} 
Consider the DPC of the hypersphere $||\underline{x}||^2=1$, given in Example \ref{standardhyp}. Equation (\ref{polarplane}) becomes $x_0=0$.
When combined with the projective normalization condition (\ref{projcond}) it leads to 
$$
x_0 = 0, \qquad\qquad x_5 =1.
$$
Now intersecting with $Q$ and taking the preimage under $\phi$ we get
$||\underline{x}||^2=1$. For the negative hypersphere $||\underline{x}||^2=-1$ we would have 
\[ 
x_5  = 0\qquad\qquad x_0+x_5  =1, 
\]
so the roles of $x_0$ and $x_5$ are interchanged and we get 
$$
x_0 = 1, \qquad\qquad x_5 = 0.
$$
Intersecting with $Q$ and pulling back under $\phi$ we get $||\underline x||^2=-1$.
\end{Ex}
\vspace{0.1in}

\begin{Ex} Consider the DPC of the hyperplane 
$\langle \underline a, \underline x\rangle = a_1x_1 + a_2x_2 - a_3x_3 - a_4x_4 = b$, which were found to be $(-b,\underline a, b)$ in Example \ref{hypplane}. Equation (\ref{polarplane}) gives
$$
-bx_0 + a_1x_1 + a_2x_2 - a_3x_3 - a_4x_4 - b x_5= 0,$$
which combined with the condition $x_0+x_5=1$ gives the desired result $\langle \underline a, \underline x\rangle = b$.
\end{Ex}
\vspace{0.1in}

%%%%%%%%%%%%%%%%%%%%%%%%%%%%%%%%%%%%%%%%%%%%%%%%%%%%%%%%%%%%%%%%%%%%%%
\subsection{Orthogonality of hyperspheres}
Let $\mathbf s, \mathbf s'\in\mathbb R^{3,3}$ be the diagonal polyspherical coordinates of two hyperspheres in $\mathbb R^{2,2}$. Recall the neutral bilinear form in the space of polyspherical coordinates given by equation (\ref{threethree}).
\vspace{0.1in}
\begin{Def} Two hyperspheres in $\mathbb R^{2,2}$ having DPC ${\mathbf s}=(s_0, \underline s, s_5)$ and ${\mathbf s'}=(s_0', \underline s', s_5')$ are said to be {\bf orthogonal} when their respective DPC are orthogonal in $\mathbb R^{3,3}$, i.e.
$$
(\mathbf s, \mathbf s'):=s_0s_0' + \langle \underline s,\underline s'\rangle - s_5s_5' =0.
$$
\end{Def}
\vspace{0.1in}
Here is the geometric interpretation of orthogonality of hyperspheres in $\mathbb R^{2,2}$: 
\vspace{0.1in}

\begin{Prop} \label{ortho} The following orthogonality conditions hold:
\begin{enumerate}[label=(\roman*)]
  \item Two proper hyperspheres 
$H_{\underline p,\pm r^2}: ||\underline x -\underline p||^2 = \pm r^2$ and $H_{\underline p',\pm r'^2}: ||\underline x -\underline p'||^2 = \pm r'^2$
are orthogonal if and only if
$||\underline p -\underline p'||^2 = \pm r^2 \pm r'^2$.

\item A proper hypersphere $H_{\underline p,\pm r^2}$ and an isotropic cone $C_{\underline p'}$ are orthogonal if and only if $\underline p'\in H_{\underline p,\pm r^2}$, i.e. $||\underline p -\underline p'||^2 = \pm r^2$.

\item Two isotropic cones $C_{\underline p}$ and $C_{\underline p'}$ are orthogonal if and only if $\underline p$ and $\underline p'$ are null-separated, i.e. $||\underline p-\underline p'||^2=0$.

\item Two hyperplanes $\langle \underline a,\underline x\rangle = b$ and $\langle \underline a',\underline x\rangle = b'$ are orthogonal if and only if $\langle \underline a,\underline a'\rangle =0.$

\item A hyperplane $\langle \underline a, \underline x\rangle =b$ and a hypersphere $H_{\underline p, \pm r^2}$ are orthogonal if and only if
the center $\underline p$ belongs to the hyperplane, i.e. $\langle \underline a, \underline p\rangle =b$
\end{enumerate}
\end{Prop}
\vspace{0.1in}

\begin{Lem}\label{distvsprod}
Let $\mathbf s_{\underline a}$ (resp $\mathbf s_{\underline a'}$) be the diagonal polyspherical coordinates of the isotropic cones at $\underline a$ (resp $\underline a'$). Then 
$$||\underline a-\underline a'||^2 = -2 (\mathbf s_{\underline a}, \mathbf s_{\underline a'}) =  -2 \left(s_0s_0'+\langle\underline{s},\underline{s'}\rangle - s_5s_5'\right).$$
In particular, two points in $\mathbb R^{2,2}$ are null-separated, i.e. $||\underline a-\underline a'||^2=0$, if and only if the DPC of its two respective isotropic cones are orthogonal in $\mathbb R^{3,3}$.
%\begin{center}
%\includegraphics[scale=0.45]{Compu}
%\end{center}
\end{Lem}
\vspace{0.1in}

%%%%%%%%%%%%%%%%%%%%%%%%%%%%%%%%%%%%%%%%%%%%%%%%%%%%%%%%%%%%%%%%%%%%%%
\subsection{Hyperplanes vs proper hyperspheres}

A closer look at the defining equation of a hypersphere (\ref{hyperspheredef}) reveals that $(s_0, \underline s, s_5)$ are the DPC of a hyperplane if and only if $s_0+s_5=0$ and $\underline s \neq 0$. Define the linear $5$-subspace $\mathcal P$ of $\mathbb R^{3,3}$ as
\begin{equation}\label{Psubspace}
\mathcal P:=\{ (s_0, \underline s, s_5)\in\mathbb R^6 \ | \ s_0 + s_5 = 0\}.
\end{equation}
The $5$-subspace $\mathcal P$ of $\mathbb R^{3,3}$ is degenerate and thus contains its complement in $\mathbb R^{3,3}$:
$$
\mathcal P^\perp = \spn{(1, \underline 0, -1)}\subset \mathcal P.
$$
Table \ref{hpshpp} clarifies the distinction between the diagonal polyspherical coordinates of proper hyperspheres and those of hyperplanes:
\vspace{0.1in}

\begin{table}[h]
\begin{tabular}{|l|l|l|}
\hline
&&\\
\textit{$\mathbb R^{2,2}$ hypersphere type} & \textit{DPC} & \textit{belongs to} \\ 
&&\\
\hline
&&\\
Proper hypersphere & $(s_0, \underline s, s_5)$ & $\mathbb R^{3,3}\setminus\mathcal P$ \\ 
&&\\
\hline
&&\\
Hyperplane & $(s_0, \underline s, -s_0)$ & $\mathcal P \setminus \mathcal P^\perp$ \\ 
&&\\
\hline
&&\\
Empty & $(s_0, \underline 0, -s_0)$ & $\mathcal P^\perp$ \\ &&\\
\hline
\end{tabular}
\vspace{0.1in}
\caption{Distinction between DPC of proper hyperspheres and hyperplanes.}
\label{hpshpp}
\end{table}
\vspace{0.1in}
%%%%%%%%%%%%%%%%%%%%%%%%%%%%%%%%%%%%%%%%%%%%%%%%%%%%%%%%%%%%%%%%%%%%%%
\subsection{Conformal transformations}
To characterize conformal transformations of pseudo-Euclidean space $\mathbb R^{2,2}$ it is convenient to regard its points as isotropic cones. 
A point $\underline x\in \mathbb R^{2,2}$ can be uniquely identified with the isotropic cone $C_{\underline x}$, having polyspherical coordinates ${\mathbf s}_{\underline x}$. This provides an embedding of the space $\mathbb R^{2,2}$ into the null cone $Q\subset \mathbb R^{3,3}$ from Definition \ref{nullconeQ}. 

In the null cone $Q$ of $\mathbb R^{3,3}$ one can find \emph{both} the diagonal polyspherical coordinates of null-hyperplanes \emph{as well as} the diagonal polyspherical coordinates of isotropic cones, so the correspondence provided is not surjective. 

Let $J:=\text{diag}(1,1,1,-1,-1,-1)$  so that $(s,s')=s^TJs'$. Define
$$
O^\pm(3,3)= \{\Lambda \in GL(6) \ | \ \Lambda^T J \Lambda = \pm J \}.
$$

\begin{Prop} \label{O33} Let $f$ be a conformal map of $\mathbb R^{2,2}$. Then there exists a linear map $\Lambda$ in $O^\pm(3,3)$ satisfying
$$
\mathbf s_{f(\underline x)} = \Lambda \mathbf s_{\underline x},
$$
for all $\underline x\in\mathbb R^{2,2}$.
\end{Prop} 
\begin{proof}
From \cite{russian} \cite{rosenfeld} we know of the existence of a linear transformation $\Lambda$ preserving the null-cone in $\mathbb R^{3,3}$ satisfying $\mathbf s_{f(\underline x)} = \Lambda \mathbf s_{\underline x}$ for all $\underline x\in\mathbb R^{2,2}$. If $\Lambda$ takes null vectors to null vectors, then by \cite{cecil} there exists a $\lambda\neq 0$ so that $\Lambda^T J \Lambda = \lambda J$. A scaling of the matrix $\Lambda$ gives the result.
\end{proof}
\vspace{0.1in}

To construct the most general conformal transformation $f$ of $\mathbb R^{2,2}$ fix a linear transformation $\Lambda\in O^\pm(3,3)$. We now explain how the image $f(\underline x)$ of an arbitrary point $\underline x$ is obtained. Let ${\mathbf s}_{\underline x}$ be the DPC of the isotropic cone $C_{\underline x}$. The vector $\Lambda {\mathbf s}_{\underline x}$ corresponds to the DPC of the isotropic cone with center $f(\underline x)$. 

In the next section the conformal geometry of lower dimensional spheres in $\mathbb R^{2,2}$ will be translated into the action of linear maps of $\mathbb R^{3,3}$ preserving the null cone $Q$. It is important to remark that these maps preserve the metric type of linear subspaces.  

\vspace{0.1in}
%%%%%%%%%%%%%%%%%%%%%%%%%%%%%%%%%%%%%%%%%%%%%%%%%%%%%%
\subsection{$m$-spheres of $\mathbb R^{2,2}$}
\vspace{0.1in}
\begin{Def} \label{mspheres} An \emph{$m$-sphere} of $\mathbb R^{2,2}$ is  defined to be the intersection of $4-m$ hyperspheres, for all $1\leq m \leq 3$. A $3$-sphere is also called a \emph{hypersphere}, and a $1$-sphere is also called a \emph{conic}.
\end{Def}
\vspace{0.1in}

There is a correspondence between $m$-spheres of $\mathbb R^{2,2}$ and linear $4-m$ subspaces $V\subset\mathbb R^{3,3}$ spanned by the diagonal polyspherical coordinates of the generating hyperspheres.

\vspace{0.1in}
\begin{Def} 
Let $S$ be an $m$-sphere given by the intersection of $4-m$ hyperspheres $H_1, H_2, \ldots, H_{4-m} \subset \mathbb R^{2,2}$ with diagonal polyspherical coordinates $\mathbf s_1, \mathbf s_2, \ldots, \mathbf s_{4-m} \in\mathbb R^{3,3}$ respectively. Define the \emph{associated $(4-m)$-subspace} $V_S\subset\mathbb R^{3,3}$ to be
$$V_S:=\spn\{\mathbf s_1, \mathbf s_2, \ldots, \mathbf s_{4-m}\}.$$
\end{Def}
\vspace{0.1in}

Table \ref{mspheres3} shows this correspondence for all possible $m$-spheres of $\mathbb R^{2,2}$:
\vspace{0.1in}

\begin{table}[h]

\begin{tabular}{|l|l|l|}
\hline
&&\\
$m$ & $S\subset \mathbb R^{2,2}$ & $V\subset \mathbb R^{3,3}$                   \\ 
&&\\
\hline
&&\\
1   & $H_1\cap H_2\cap H_3$         & $\spn\{\mathbf s_1, \mathbf s_2, \mathbf s_3\}$ \\ 
&&\\
\hline
&&\\
2   & $H_1\cap H_2$                 & $\spn\{\mathbf s_1, \mathbf s_2\}$              \\ 
&&\\
\hline
&&\\
3   & $H_1$                        & $\spn\{\mathbf s_1\}$                          \\ 
&&\\
\hline
\end{tabular}
\vspace{0.1in}
\caption{Correspondence between $m$-spheres in $\mathbb R^{2,2}$ and linear subspaces in $\mathbb R^{3,3}$. }
\label{mspheres3}
\end{table}
\vspace{0.1in}

\begin{Prop} \label{welldefasssubspace}
The associated subspace of an $m$-sphere $S$ is well-defined, in the sense that if $S$ is obtained as intersection of hyperspheres in two different ways, 
$$S = \bigcap_{i\in I} H_i = \bigcap_{i\in I} H_i', \qquad I=\{1,\ldots, 4-m\},$$ 
then the span of the diagonal polyspherical coordinates $\mathbf s_i$ of $H_i$ agrees with the span of the diagonal polyspherical coordinates $\mathbf s_i'$ of $H_i'$. 
\end{Prop}
\vspace{0.1in}
\begin{Cor}
Every choice of basis choice of basis of a subspace $V\subset \mathbb R^{3,3}$ gives a different way to obtain the same $m$-sphere as an intersection of different hyperspheres.
\end{Cor}
\vspace{0.1in}

\begin{Prop} \label{propasssubspace} The following properties of the associated subspace hold:

\begin{enumerate}
\item The associated subspace of an $m$-sphere respects the incidence relation between lower dimensional spheres in $\mathbb R^{2,2}$. That is, if $S$ is an $m$-sphere and $S'$ an $m'$-sphere with $m\leq m'$ then $S\subset S'$ if and only if $V_{S'} \subset V_S$. 

\item Let $S$ be an $m$-sphere with associate subspace $V\subset\mathbb R^{3,3}$, and let $\underline x\in\mathbb R^{2,2}$. Then the point $\underline x\in S$ if and only if the polyspherical coordinates $\mathbf s_{\underline x}$ of the cone $C_{\underline x}$ are orthogonal to $V$, i.e.  ${\mathbf s}_{\underline x}\in V^\perp.$
\end{enumerate}
\end{Prop}
\vspace{0.1in}

For the rest of this section we focus on the description of conics. The following type of conics will be important in our discussion:    
\vspace{0.1in}

\begin{Def}
A $1$-sphere or conic is \emph{non-degenerate} whenever its associated $3$-subspace is an indefinite subspace in $\mathbb R^{3,3}$.
\end{Def}
\vspace{0.1in}

\begin{Ex}\label{S0V0}
Consider the standard circle $S_0$ defined by the intersection of the hyperspheres $H: H_{\underline 0, 1}=\{||\underline x||^2 = 1$\}, $H': x_3=0$ and $H'': x_4=0$ with DPC $\mathbf s = (1, 0,0,0,0, 0)$, $\mathbf s' = (0, 0, 0 ,1, 0, 0)$ and $\mathbf s'' = (0, 0,0,0,1, 0)$ respectively. Its associated $3$-subspace $V_0$ is given by
\begin{equation}\label{V0}
V_0=\spn\{(1, 0,0,0,0, 0),(0, 0, 0 ,1, 0, 0),(0, 0,0,0,1, 0)\}.
\end{equation}
The subspace $V_0\subset \mathbb R^{3,3}$ has indefinite signature $(+--)$, so it follows that $S_0$ is non-degenerate.

Similarly, it can be seen that the associated $3$-subspace $V_0^\perp$ of $S_0^\perp$ is given by
\begin{equation}\label{V0perp}
V_0^\perp=\spn\{(0, 0,0,0,0, 1),(0, 1, 0, 0, 0, 0),(0, 0, 1,0,0, 0)\}.
\end{equation}
The subspace $V_0^\perp\subset \mathbb R^{3,3}$  has indefinite signature $(-++)$ and hence $S_0^\perp$ is a non-degenerate conic. 

%From $V_0$ one can obtain the points in $S_0$ by: 
%\begin{itemize}
%\item intersecting the complementary $3$-space $V_0^\perp = \{s_0=s_3=s_4=0\}$ with the null cone
%\item imposing the projective normalization condition 
%\end{itemize}
\end{Ex}
\vspace{0.1in}

\begin{Lem} \label{vcapp}
A conic $S$ in $\mathbb R^{2,2}$ is generically contained in a flat $2$-sphere $\pi$ (also called affine $2$-plane) in $\mathbb R^{2,2}$ whose associated $2$-subspace is determined by the intersection $V_S\cap \mathcal P$.
\end{Lem}
\begin{proof}
The associated indefinite $3$-subspace $V_S$ and the degenerate $5$-subspace $\mathcal P$ defined in equation (\ref{Psubspace}) intersect on a $2$-subspace of $\mathbb R^{3,3}$. The subspace $V_S\cap \mathcal P$ can be interpreted as the associated subspace $V_\pi$ of a $2$-sphere $\pi$ in $\mathbb R^{2,2}$. According to Table \ref{hpshpp} the $2$-sphere $\pi$ can be put as the intersection of two hyperplanes in $\mathbb R^{2,2}$, and hence $\pi$ is a \emph{flat} $2$-sphere, or an affine $2$-plane in $\mathbb R^{2,2}$. It follows by Proposition \ref{propasssubspace} that $S\subset \pi$ as the associated subspaces satisfy the reverse inclusion $V_S\cap\mathcal P\subset V_S$.
\end{proof}
\vspace{0.1in}

\begin{Lem} \label{sgn2plane} An affine $2$-plane in $\mathbb R^{2,2}$ is definite/indefinite/degenerate if and only if its associated subspace is definite/indefinite/degenerate in $\mathbb R^{3,3}$.
\end{Lem}
\vspace{0.1in}

%%%%%%%%%%%%%%%%%%%%%%%%%%%%%%%%%%%%%%%%%%%%%%%%%%
\subsection{Conjugate conics in $\mathbb R^{2,2}$}
\vspace{0.1in}

\begin{Def} \label{def:ndccs}
A pair of {\it non-degenerate conjugate conics} $S,S^\perp$ in $\mathbb R^{2,2}$ is the image of the standard circles $S_0,S_0^\perp$ under a conformal map of $\mathbb R^{2,2}$. 
\end{Def}
\vspace{0.1in}

Since conformal maps preserve hyperspheres,  non-degenerate conjugate conics as introduced in Definition \ref{def:ndccs} are always pairs of \emph{conics}, in the sense of Definition \ref{mspheres}.
\vspace{0.1in}

\begin{Prop} \label{ndccs2} 
Two $1$-spheres $S,S^\perp$ are a pair of \emph{non-degenerate conjugate conics} if and only if their corresponding $3$-subspaces $V,V^\perp$ are indefinite and complementary in $\mathbb R^{3,3}$.
\end{Prop}
\begin{proof} 
It can be checked from Example \ref{S0V0} that the associated $3$-subspaces $V_0, V_0^\perp$ of the standard conjugate circles $S_0,S_0^\perp$ are indefinite and complementary in $\mathbb R^{3,3}$. Let $S,S^\perp$ be obtained from $S_0,S_0^\perp$ by a conformal map $f$ of $\mathbb R^{2,2}$, and let $V,V^\perp$ be the respective associated $3$-subspaces. The associated linear map $\Lambda\in O^\pm(3,3)$ maps $V_0$ to $V$, and $V_0^\perp$ to $V^\perp$. It follows that $V,V^\perp$ will also be indefinite and complementary in $\mathbb R^{3,3}$.
The converse follows by the transitivity of $O^\pm(3,3)$ on indefinite $3$-subspaces.
\end{proof}
\vspace{0.1in}

We are still left to prove the classification result for non-degenerate conjugate conics in terms of their containing planes, which is the content of Theorem \ref{equivdefs}. According to Lemma \ref{vcapp}, the unique plane containing a conic comes described by the intersection of the associated $3$-subspace relative to the conic with $\mathcal P$. The following is a preliminary technical result:
\vspace{0.1in}

\begin{Lem} \label{technicallem} Let $\mathcal P$ be the subspace of $\mathbb R^{3,3}$ given by $s_0+s_5=0$ and $V$ an indefinite $3$-subspace. Then there are the following possibilities for the intersection between $\mathcal P$ and the complementary indefinite subspaces $V,V^\perp$:
\begin{itemize}
\item {\bf Case 0}. $\mathcal P^\perp \subset V,V^\perp$ and $V,V^\perp \subset \mathcal P$.
\item {\bf Case 1}. $V\subset \mathcal P$ and $\mathcal P^\perp \subset V^\perp \cap \mathcal P$, or $V^\perp\subset \mathcal P$ and $\mathcal P^\perp \subset V \cap \mathcal P$
\item {\bf Case 2}. Both $V,V^\perp$ intersect $\mathcal P$ in a $2$-subspace, none of which contains $\mathcal P^\perp$. Then $V\cap \mathcal P$ is a definite/indefinite/degenerate $2$-subspace if and only if  $V^\perp \cap \mathcal P$ is a definite/indefinite/degenerate $2$-subspace.
\end{itemize}\end{Lem}
\begin{proof} 
Assume first $V\subset \mathcal P$. It follows that $\mathcal P^\perp \subset V^\perp$.  Assuming $\mathcal P^\perp \subset V$ one gets Case 0, as $V^\perp \subset \mathcal P^{\perp\perp}=\mathcal P$. If, on the other hand, $\mathcal P^\perp \not\subset V$ one has $V^\perp\not\subset \mathcal P$ and it follows that $V^\perp$ intersects $\mathcal P$ on a $2$-subspace containing $\mathcal P^\perp$, which corresponds to Case 1.

Assume now $\mathcal P^\perp \not \subset V\not \subset \mathcal P$.  It follows that $\mathcal P^\perp \not \subset V^\perp\not \subset \mathcal P$. Hence both $V, V^\perp$ intersect $\mathcal P$ on respective $2$-subspaces none of which contains $\mathcal P^\perp$. The subspace $\mathcal P$ then can be written as the direct sum $\mathcal P=(V\cap \mathcal P) \oplus (V\cap \mathcal P^\perp) \oplus \mathcal P^\perp$. It follows that the signature $\sgn \mathcal P$ can be decomposed as
$$
\sgn \mathcal P = \sgn(V\cap \mathcal P) + \sgn(V^\perp\cap \mathcal P) + \sgn \mathcal P^\perp, 
$$
where $\sgn$ takes the difference between the number of positive and negative eigenvalues of the induced metric in $\mathbb R^{3,3}$. 
It follows that:
\begin{equation}\label{sgnrel}
\sgn(V\cap \mathcal P) + \sgn(V^\perp\cap \mathcal P) = 0.
\end{equation}
From equation (\ref{sgnrel}) one can see that $V\cap \mathcal P$ is definite/indefinite/degenerate if and only if $V^\perp \cap \mathcal P$ is definite/indefinite/degenerate, which makes up Case 2.
\end{proof}
\vspace{0.1in}

\noindent{\it Proof of Theorem \ref{equivdefs}}

A $3$-subspace $V$ containing  $\mathcal P^\perp$ does not correspond to any conic in $\mathbb R^{2,2}$, hence Case 0 from Lemma \ref{technicallem} is ruled out. We look at the remaining cases:
\begin{itemize}
\item[{\bf Case 1}.] Assume  $\mathcal P^\perp \not\subset V\subset \mathcal P$ and $\mathcal P^\perp \subset V^\perp$. Then $V=V_S$ is the associated $3$-space of a conic $S$ which can be obtained as the intersection of three hyperplanes in $\mathbb R^{2,2}$ by Table \ref{hpshpp}. Thus, $S$ is a line and $S^\perp=$\O.

\item [{\bf Case 2}.] Assume both $V\cap \mathcal P$ and $V^\perp\cap \mathcal P$ are $2$-dimensional subspaces. Let $\mathbf s, \mathbf s', \mathbf s''$ and $\mathbf s_\perp,  \mathbf s'_\perp, \mathbf s''_\perp$ be two bases of $V$ and $V^\perp$, respectively, so that $\mathbf s, \mathbf s'\in V\cap \mathcal P$ and $\mathbf s_\perp, \mathbf s'_\perp\in V^\perp\cap \mathcal P$. 

The vectors $\mathbf s$ and $\mathbf s'$ can be written as $\mathbf s=(-b, \underline a, b)$ and $\mathbf s'=(-b', \underline a', b')$ for some $\underline a, \underline a'\in\mathbb R^{2,2}$ and $b,b'\in\mathbb R$ and so they represent two hyperplanes $H: \langle \underline a, \underline x\rangle = b$ and $H': \langle \underline a', \underline x\rangle = b'$ of $\mathbb R^{2,2}$. Similarly, the vectors $\mathbf s_\perp=(-\beta, \underline \alpha, \beta)$ and $\mathbf s_\perp'=(-\beta', \underline \alpha', \beta')$ represent two hyperplanes $H_\perp: \langle \underline \alpha, \underline x\rangle = \beta$ and $H_\perp': \langle \underline \alpha', \underline x\rangle = \beta'$ of $\mathbb R^{2,2}$. Define the $2$-planes in $\mathbb R^{2,2}$
$\pi:=H\cap H'$ and $\pi^\perp:=H_\perp\cap H_\perp'$. The metric type of the affine planes $\pi, \pi^\perp$ is given in each case by Lemma \ref{sgn2plane}.

Finally, as the vectors $\mathbf s'', \mathbf s''_\perp$ are not in $\mathcal P$, they are the diagonal polyspherical coordinates of two proper hyperspheres which we will assume to be $H_{\underline p, \pm r^2}$ and $H_{\underline {\tilde p}, \pm \tilde r^2}$, respectively.

\begin{enumerate}[label=(\roman*)]
\item If both $V\cap \mathcal P$ and $V^\perp\cap \mathcal P$ are definite then the basis formed by $\mathbf s, \mathbf s',\mathbf s'', \mathbf s_\perp, \mathbf s'_\perp,\mathbf s''_\perp$ can be taken to be an orthonormal basis of $\mathbb R^{3,3}$. 

%have inner products in $\mathbb R^{3,3}$ given by the matrix $\diag(1,1,-1,-1,-1,1)$.

%the choosen bases for $V$ and $V^\perp$ can be taken to be orthogonal in $\mathbb R^{3,3}$.

%The vector $s''$ is not in $\mathcal P$ so it corresponds to the polyspherical cooridnates of some hypersphere $H_{\underline p, r^2}$ in $\mathbb R^{2,2}$ with $r^2\neq 0$ (as $s''$ is not null). 

%As $V\cap \mathcal P$ is definite  then so is the plane $\pi:= H\cap H'$. The orthogonality between $s''$ and $s,s'$ means that $\underline p\in \pi$. 

%The vector $s''_\perp$ represents a hypersphere $H_{\tilde{ \underline p}, \tilde r^2}$ with $\tilde r^2 \neq 0$. The assumptions on $s_\perp, s'_\perp, s''_\perp$ imply that the plane $\pi^\perp:= H_\perp\cap H_\perp'$ is definite and $\tilde{\underline p}\in \pi^\perp$. 

Applying orthogonality conditions from Proposition \ref{ortho} to $\mathbf s, \mathbf s'$ and $\mathbf s_\perp, \mathbf s'_\perp$ we get that $\pi$ and $\pi^\perp$ are orthogonal in $\mathbb R^{2,2}$ and hence must intersect at a point. Using orthogonality between $\mathbf s''$ and $\mathbf s_\perp, \mathbf s'_\perp$ one gets $\underline p\in \pi\cap\pi^\perp$, and similarly for $\tilde{\underline p}\in\pi\cap\pi^\perp$. It follows that $\underline p=\tilde{\underline p}$. Orthogonality between $\mathbf s''$ and $\mathbf s_\perp''$ will give $\pm r^2 = \mp \tilde r^2$. Finally, the radius-square cannot be zero, as the vector $\mathbf s''$ was taken to be non-null.

\item If $V\cap \mathcal P$ and $V^\perp\cap \mathcal P$ are indefinite a similar argument follows. 

\item If $V\cap \mathcal P$ and $V^\perp\cap \mathcal P$ are degenerate one can choose $\mathbf s$ to be null, $\mathbf s'$ to be non-null and orthogonal to $\mathbf s$, and $\mathbf s''$ to be null and not orthogonal to $\mathbf s$ (apply a similar choice for $\mathbf s_\perp, \mathbf s_\perp', \mathbf s_\perp''$). As $\mathbf s''$ and $\mathbf s''_\perp$ are null, they are the diagonal polyspherical coordinates of isotropic cones $C_{\underline p}$ and $C_{\underline{\tilde p}}$, respectively.

We now see that the degenerate affine planes $\pi$ and $\pi^\perp$ do not intersect. The vectors $\underline a, \underline \alpha$ are null and orthogonal in $\mathbb R^{2,2}$. However, there exist non-null orthogonal vectors to both of them (e.g $\underline a'$) and so $\underline a, \underline \alpha$ must be linearly dependent (two orthogonal null vectors in an $\mathbb R^{2,1}$ are linearly dependent). As  diagonal polyspherical coordinates are defined up to scalar one can assume $s_\perp = (-\beta, \underline a, \beta)$ and it follows that $b\neq \beta$ as otherwise $V$ and $V^\perp$ would have non-zero intersection. Hence the affine planes $\pi, \pi^\perp$ do not intersect in $\mathbb R^{2,2}$.  

From Proposition \ref{ortho}, orthogonality between $\mathbf s''$ and $\mathbf s_\perp, \mathbf s'_\perp$ implies $\underline p\in \pi^\perp$, and by a similar argument $\tilde{\underline p} \in \pi$. Finally, as $\mathbf s''$ and $\mathbf s_\perp''$ are orthogonal, the points $\underline p$ and $\tilde{\underline p}$ will be null-separated.
\end{enumerate}
\end{itemize}
\qed
\vspace{0.1in}

%%%%%%%%%%%%%%%%%%%%%%%%%%%%%%%%%%%%%%%%%%%%%%%%%%%%%%%%%%%%%%%%%%%%%%
\section{Conjugate conics as lines of a doubly ruled surfaces in $\mathbb R^3$}
%%%%%%%%%%%%%%%%%%%%%%%%%%%%%%%%%%%%%%%%%%%%%%%%%%%%%%%%%%%%%%%%%%%%%%

In this section, the space $\mathbb R^{2,2}$ is interpreted as a local chart of the space of lines in $3$-space. For a line having  Plücker coordinates $p_1,p_2,p_3,q_1,q_2,q_3$ define the following coordinates
\begin{equation}\label{johnscoords}
\underline x = \left( \frac{p_2+q_2}{q_3} , \frac{-p_1-q_1}{q_3}, 
\frac{p_2-q_2}{q_3},
\frac{-p_1+q_1}{q_3}
\right)\in\mathbb R^{2,2}.
\end{equation}

Observe that both horizontal lines (those for which $q_3=0$) and improper lines are not covered by these lines coordinates, i.e. they appear to be \emph{at infinity} in $\mathbb R^{2,2}$. The following important result explains the connection between projective line geometry and the geometry of the pseudo-Euclidean space $\mathbb R^{2,2}$:

\vspace{0.1in}
\begin{Prop} 
Two lines in $3$-space intersect or are parallel if and only if their corresponding vectors $\underline x, \underline x'\in\mathbb R^{2,2}$ defined as in equation (\ref{johnscoords}) are null-separated (when defined), i.e. $||\underline x - x'||^2 = 0$. 
\end{Prop}
\vspace{0.1in}

In this section, the null-separation properties between the points in non-degenerate conjugate conics are used to show their relation to the two families of generating lines of a doubly ruled surface in $3$-space.
\vspace{0.1in}

%%%%%%%%%%%%%%%%%%%%%%%%%%%%%%%%%%%%%%%%%%%%%%%%%%%%%%%%%%%%%%%%%%%%%%%%
\subsection{Null-separation property of non-degenerate conjugate conics}
\vspace{0.1in}
\begin{Lem} \label{threecones}
Three isotropic cones in $\mathbb R^{2,2}$ intersect on a non-degenerate conic if and only if the centers are not null-separated.
\end{Lem}
\begin{proof} Let $\underline a, \underline a', \underline a''\in\mathbb R^{2,2}$ be three non-null-separated points. Let $\mathbf s, \mathbf s',\mathbf s'' \in\mathbb R^{3,3}$ be the diagonal polyspherical coordinates of the isotropic cones $C_{\underline a}, C_{\underline a'}$ and $C_{\underline a''}$, respectively. From the non-null-separation of the centers it follows from Lemma \ref{distvsprod} that $$
(\mathbf s, \mathbf s'), (\mathbf s',\mathbf s''), (\mathbf s,\mathbf s'') \neq 0.
$$
As $\mathbf s, \mathbf s',\mathbf s''$ are DPC of isotropic cones we also have the nullity condition
$$
(\mathbf s,\mathbf s)= (\mathbf s',\mathbf s')=(\mathbf s'',\mathbf s'')=0.
$$
Let $G=(\mathbf s^{i}, \mathbf s^{j})$. Then:
\begin{align*}
\Tr G & = 0 \\
\Det G & = 2 (\mathbf s, \mathbf s') (\mathbf s',\mathbf s'') (\mathbf s,\mathbf s'') \neq 0.
\end{align*}

Let $S$ be the conic determined by the intersection of the three null cones. The associated $3$-subspace $V=\spn\{\mathbf  s, \mathbf s', \mathbf s''\}$ is non-degenerate and has indefinite signature, and thus $S$ is non-degenerate. Moreover, if ones relaxes the non-null-separation of the centers, the associated $3$-subspace $V$ becomes degenerate, hence the converse also holds.
\end{proof}
\vspace{0.1in}

\begin{Prop}\label{3pointscondition} Two curves $S,S^\perp$ in $\mathbb R^{2,2}$ are non-degenerate conjugate conics if and only if for all  non-null-separated points $\underline a, \underline a', \underline a''\in S$ and  non-null-separated points $\underline b,\underline b', \underline b''\in S^\perp$ we have 
$$
S =C_{\underline b} \cap C_{\underline b'} \cap C_{\underline b''},
$$
and
$$
S^\perp = C_{\underline a} \cap  C_{\underline a'} \cap C_{\underline a''}.
$$
\end{Prop}
\begin{proof}
Let $S,S^\perp$ be a pair of non-degenerate conjugate conics with respective associated $3$-subspaces $V,V^\perp$, and choose any three non-null-separated points $\underline a, \underline a', \underline a''\in S$. By Proposition \ref{propasssubspace} the vectors $\mathbf s_{\underline a}, \mathbf s_{\underline a'}, \mathbf s_{\underline a''}$ defined as in equation (\ref{eq:DPScone}) belong to the subspace $V^\perp$. Moreover, as the points are non-null-separated in $\mathbb R^{2,2}$, the corresponding vectors form a null basis of $V^\perp$. A point $\underline x\in\mathbb R^{2,2}$ will lie in the intersection $C_{\underline a} \cap  C_{\underline a'} \cap C_{\underline a''}$ if and only if the vector $\mathbf s_{\underline x}$ is orthogonal to $\mathbf s_{\underline a}, \mathbf s_{\underline a'}, \mathbf s_{\underline a''}$ in $\mathbb R^{3,3}$. Hence $\mathbf s_{\underline x}\in (V^\perp)^\perp=V$ and it follows that $\underline x\in S^\perp$. A similar argument shows that $S$ is obtained as the intersection of three isotropic cones centered at any non-null-separated three points in $S^\perp$. 

We now prove the converse. Let $\underline a, \underline a', \underline a''$ be non-null-separated points in the curve $S$. It follows from Lemma \ref{threecones} that $S^\perp=C_{\underline a} \cap  C_{\underline a'} \cap C_{\underline a''}$ is a non-degenerate conic with associated indefinite $3$-subspace $V^\perp = \spn\{\mathbf s_{\underline a}, \mathbf s_{\underline a'}, \mathbf s_{\underline a''}\}$. Similarly, choosing three non-null-separated points $\underline b, \underline b', \underline b''\in S^\perp$ we find that $S$ is a non-degenerate conic with associated indefinite $3$-subspace $V=\spn\{\mathbf s_{\underline b}, \mathbf s_{\underline b'}, \mathbf s_{\underline b''}\}$. As the points $\underline a, \underline a', \underline a''$ are null-separated from the points $\underline b, \underline b', \underline b''$ in $\mathbb R^{2,2}$, it follows from Lemma
\ref{distvsprod} that the vectors $\mathbf s_{\underline a}, \mathbf s_{\underline a}, \mathbf s_{\underline a}$ are pairwise orthogonal to $\mathbf s_{\underline b}, \mathbf s_{\underline b'}, \mathbf s_{\underline b''}$ in $\mathbb R^{3,3}$. Hence the spaces $V$ and $V^\perp$ are complementary in $\mathbb R^{3,3}$, and it follows by Proposition \ref{ndccs2} that $S$ and $S^\perp$ are a pair of non-degenerate conjugate conics.
\end{proof}
\vspace{0.1in}

\begin{Cor} (Null-distance property between points of non-degenerate conjugate conics). Given a pair $S,S^\perp$ of non-degenerate conjugate conics and two points $\underline a \in S$ and $\underline b \in S^\perp$ then $||\underline a-\underline b||^2=0$.
\end{Cor}
\vspace{0.1in}

%%%%%%%%%%%%%%%%%%%%%%%%%%%%%%%%%%%%%%%%%%%%%%%%%%%%%%%%%%%%%%%%%%%%%%
\subsection{Relation to lines of a doubly ruled surface}
\vspace{0.1in}
\begin{Def} A \emph{non-planar doubly ruled surface} is either a one-sheeted hyperboloid or a hyperbolic paraboloid. The two foliations by lines that these surfaces admit are called the two \emph{families of generating lines} of the surface.
\end{Def}

We use the characterization of the lines of a doubly ruled surface from \cite{hcv}. The locus of lines incident to three skew lines in space is a family of lines $L$ of a doubly ruled surface (a one-sheeted hyperboloid when the original lines are in general position, and a hyperbolic paraboloid when they are parallel to a fixed plane). The other family of lines $L^\perp$ is obtained as the locus of lines incident to any three lines from $L$. Note in particular how this means that given any three skew lines in space, there is a unique doubly ruled surface containing those lines.
\vspace{0.1in}

\begin{Prop} \label{3linescondition} Two one-parameter families of lines $L,L^\perp$ in $3$-space are the two families of generating lines of a non-planar doubly ruled surface if and only if for every three skew lines $r,r',r''\in L$ and three skew lines $s,s',s''\in L^\perp$ we have

\begin{center}
$L$ is the locus of lines incident to $s,s'$ and $s''$,
\end{center}
and
\begin{center}
$L^\perp$ is the locus of lines incident to $r,r'$ and $r''$.
\end{center}
\end{Prop}

\vspace{0.1in}
\noindent{\it Proof of Theorem \ref{rulsurf}}

Let $S,S^\perp$ be a pair of non-degenerate conjugate conics in $\mathbb R^{2,2}$ and let $L,L'$ be the lines representing them.
The condition from Proposition \ref{3pointscondition} is equivalent to the condition from Proposition \ref{3linescondition}, as null-separation of points in $\mathbb R^{2,2}$ is the same as incidence of lines they represent in $3$-space. Hence, $S,S^\perp$ are non-degenerate conjugate conics iff they satisfy the condition from Proposition \ref{3pointscondition} iff $L,L^\perp$ satisfy the condition from Proposition \ref{3linescondition} iff $L,L'$ are the two families of a non-planar doubly ruled surface in $3$-space.
\qed
\vspace{0.1in}

%%%%%%%% Bibliography

\end{document}